\documentclass[11pt,reqno, oneside]{amsart}
\usepackage[hidelinks]{hyperref}
\usepackage{caption}
\usepackage{amssymb,amsfonts,amscd,amsbsy, color, amsmath}
\usepackage[mathscr]{eucal}
\usepackage{url}
\usepackage{graphicx}
\usepackage{mathtools}
\usepackage[textwidth=25mm,textsize=tiny]{todonotes}
\usepackage{tabularx}
\usepackage{float, color}
\usepackage{placeins}
\usepackage{enumitem}
\usepackage{cite}
\usepackage{soul}

\newtheorem{theorem}{Theorem}[section]
\newtheorem{lemma}[theorem]{Lemma}
\newtheorem{proposition}[theorem]{Proposition}

\theoremstyle{definition}

\newtheorem{remark}[theorem]{Remark}
\numberwithin{equation}{section}

\newcommand{\bea}{\begin{equation}\begin{aligned}}
\newcommand{\eea}{\end{aligned}\end{equation}}

\makeatletter
\@namedef{subjclassname@2020}{
  \textup{2020} Mathematics Subject Classification}
\makeatother

\begin{document}

\title[Generalized rank deviations for overpartitions]{Generalized rank deviations for overpartitions}

\author{Kevin Allen}
\author{Robert Osburn}
\author{Matthias Storzer}

\address{School of Mathematical Sciences, University College Cork, Cork, Ireland}

\email{kallen@ucc.ie}
\email{robert.osburn@ucc.ie}
\email{mstorzer@ucc.ie}


\date{\today}

\begin{abstract}
We prove formulas for generalized rank deviations for overpartitions. These formulas are in terms of Appell--Lerch series and sums of quotients of theta functions and extend work of Lovejoy and the second author. As an application, we compute a dissection.
\end{abstract}

\maketitle
\section{Introduction}

A partition of a natural number $n$ is a non-increasing sequence of positive integers whose sum is $n$. An overpartition of $n$ is a partition in which the first occurrence of each distinct part may be overlined. For example, the partitions and overpartitions of $4$ are
\begin{equation*}
\begin{gathered}
4, \, 3+1, \, 2+2, \, 2+1+1, \, 1+1+1+1
\end{gathered}
\end{equation*}
and
\begin{equation*}
\begin{gathered}
4, \, \overline{4}, \, 3+1, \, \overline{3} + 1, \, 3 + \overline{1}, \,
\overline{3} + \overline{1}, \, 2+2, \overline{2}
+ 2, 2+1+1, \, \overline{2} + 1 + 1, \, 2+ \overline{1} + 1, \\
\overline{2} + \overline{1} + 1, \, 1+1+1+1, \, \overline{1} + 1 + 1 +1,
\end{gathered}
\end{equation*}
respectively. For a curated survey which highlights the importance of overpartitions in $q$-series, number theory and algebraic combinatorics, see \cite{cl}. The rank of a partition $\lambda$ is the largest part $\ell(\lambda)$ minus the number of parts $\#(\lambda)$. For overpartitions, one can consider two statistics, the rank or the $M_2$-rank which is defined by \cite{love2}
$$
\text{$M_2$-rank} \hspace{.025in} (\pi) = \bigg \lceil \frac{\ell(\pi)}{2}
\bigg \rceil - \#(\pi) + \#(\pi_o) - \chi(\pi)
$$
where $\pi_o$ is the subpartition consisting of the odd non-overlined parts and 
$$
\chi(\pi) := \chi(\text{the largest part of $\pi$ is odd and non-overlined}). 
$$
Throughout, we use the standard notation $\chi(X):=1$ if $X$ is true and $0$ if $X$ is false. A recent topic of interest for these combinatorial objects and their enumerative data is the study of rank deviations \cite{hm2, loover}. For integers $0 \leq a \leq M$ where $M \geq 2$, consider
\begin{equation} \label{dev1}
\overline{D}(a, M) := \sum_{n \geq 0} \biggl( \overline{N}(a, M, n) - \frac{\overline{p}(n)}{M} \biggr) q^n
\end{equation}
and
\begin{equation} \label{dev2}
\overline{D}_{2}(a,M) := \sum_{n \geq 0} \biggl( \overline{N}_{2}(a, M, n) - \frac{\overline{p}(n)}{M} \biggr) q^n
\end{equation}
where $\overline{N}(a, M, n)$ denotes the number of overpartitions of $n$ with rank congruent to $a$ modulo~$M$,  $\overline{N}_{2}(a, M, n)$ denotes the number of overpartitions of $n$ with $M_2$-rank congruent to $a$ modulo~$M$ and $\overline{p}(n)$ is the number of overpartitions of $n$. Some impetuses for finding explicit formulas for the rank deviations (\ref{dev1}) and (\ref{dev2}) are to correct the literature \cite[Remark 1.6]{loover}, provide a general framework in which one can recover all known rank difference identities for overpartitions \cite[Section 4]{loover} and prove the modularity of rank generating functions in arithmetic progressions (e.g., \cite[Section 7]{hm2}) thereby generalizing \cite{bl, dewar, rm}. 

For an integer $d \geq 1$, consider the generating function
\begin{equation} \label{gen}
\begin{aligned}
\mathcal{O}_{d}(z;q) & := \frac{(-q)_{\infty}}{(q)_{\infty}} \Biggl(1 + 2 \sum_{n \geq 1} \frac{(1-z)(1-z^{-1}) (-1)^n q^{n^2 + dn}}{(1-zq^{dn})(1-z^{-1} q^{dn})} \Biggr) \\
& =: \sum_{\substack{m \in \mathbb{Z} \\ n \geq 0}} \overline{N}_{d}(m,n) z^m q^n
\end{aligned}
\end{equation}
where
\begin{equation*}
(x)_{\infty} = (x;q)_{\infty} := \prod_{k=0}^{\infty} (1-xq^{k}).
\end{equation*} 
The $d=1$ case of (\ref{gen}) gives the two-variable generating function for the rank for overpartitions \cite{love1} while the $d=2$ case yields the two-variable generating function for the $M_2$-rank for overpartitions \cite{love2}. The mock and quantum modular properties of (\ref{gen}) were the focus of \cite{dfnsx ,jss}. Motivated by \cite{dfnsx, jss, loover}, our aim is to find explicit formulas for the generalized rank deviations for overpartitions
\begin{equation} \label{dev}
\overline{D}_{d}(a, M) := \sum_{n \geq 0} \biggl( \overline{N}_d(a,M,n) - \frac{\overline{p}(n)}{M} \biggr) q^n
\end{equation}
where 
\begin{equation} \label{Nd}
\overline{N}_d(a,M,n):= \sum_{k \, \equiv \, a \;(\text{mod}\;M)} \overline{N}_{d}(k,n).
\end{equation}

To state our main results, we require some setup. Consider the Appell--Lerch series
\begin{equation} \label{al}
m(x,q,z) := \frac{1}{j(z;q)} \sum_{r \in \mathbb{Z}} \frac{(-1)^r q^{\binom{r}{2}} z^r}{1 - q^{r-1} xz}.
\end{equation}
Here, $x$, $q$ and $z$ are non-zero complex numbers with $| q | < 1$, neither $z$ nor $xz$ is an integral power of $q$ and
\begin{equation} \label{j}
j(z;q) := (z)_{\infty} (q/z)_{\infty} (q)_{\infty} = \sum_{n \in \mathbb{Z}} (-1)^n q^{\binom{n}{2}} z^n.
\end{equation} 
Let
\begin{equation} \label{delta}
\Delta(x,z_1, z_0;q) := \frac{z_0 J_1^3 j(z_1/z_0;q) j(x z_0 z_1; q)}{j(z_0;q) j(z_1;q) j(xz_0;q) j(xz_1;q)}
\end{equation}
and
\begin{equation} \label{Psikndef}
\Psi_{k}^{n}(x,z,z';q) := -\frac{x^k z^{k+1} J_{n^2}^3}{j(z;q) j(z';q^{n^2})} \sum_{t=0}^{n-1} \frac{q^{\binom{t+1}{2} + kt} (-z)^t j(-q^{\binom{n+1}{2} + nk + nt} (-z)^n / z', q^{nt} (xz)^n z'; q^{n^2})}{j(-q^{\binom{n}{2} - nk} (-x)^n z', q^{nt} (xz)^n; q^{n^2})}
\end{equation}
where $J_{m} := (q^m; q^m)_{\infty}$ and $j(z_1,z_2;q) := j(z_1;q)j(z_2;q)$. We use the term ``generic" to mean that the parameters do not cause poles in the Appell--Lerch series or in the quotients of theta functions. Finally, for odd $d \geq 1$ and generic $z$, $z_0$ and $z' \in \mathbb{C}^{*}$, let
\begin{equation} \label{genlam}
\begin{aligned}
\Lambda(d, z, z_0, z') &:= (-1)^{\frac{d+1}{2}} q^{-\frac{(d-1)^2}{4}} z^{\frac{d-1}{d}} \Bigl( \Psi_{\frac{d-1}{2}}^{d}(z^{-\frac{2}{d}} q^d, z_0, z'; q^2) \\
& \qquad \qquad \qquad \qquad \qquad \qquad \qquad \quad + \frac{1}{d} \sum_{t=0}^{d-1} \zeta_{d}^{-t} \Delta(\zeta_d^{-2t} z^{-\frac{2}{d}} q^d, \zeta_d^{t} z^{\frac{1}{d}} q^{-\frac{d-1}{2}}, z_0; q^2) \Bigr).
\end{aligned}
\end{equation}
Here and throughout, $\zeta_M$ denotes a primitive root of unity of order $M$.

We can now state our two main results in which explicit formulas are given for the pair of deviations $\overline{D}_d(a, M) + \overline{D}_{d}(a-1, M)$. Theorem \ref{main1} addresses the case of odd $d$
and is slightly more involved with the appearance of the additional term (\ref{genlam}). Theorem \ref{main2} treats the case of even $d$ and is cleaner. As discussed in Section 4, there is no loss in generality in computing pairs of deviations. Our results can be used to find a formula for any single $\overline{D}_{d}(a, M)$. 

\begin{theorem} \label{main1} Let $d \geq 1$ be an odd integer and $2 \leq a \leq M$. For generic $z'$, $z''$ and $z_0 \in \mathbb{C}^{*}$, we have the following generating functions: 

\begin{itemize}
\item[($i$)] If $a$ and $M$ are even, then
\begin{equation} \label{e1}
\begin{aligned}
\overline{D}_d(a, M) + \overline{D}_{d}(a-1, M)  \, &=\,  \chi(a=M) \\
& +\, 2(-1)^{\frac{a}{2}} q^{-\frac{a^2}{4} + \frac{a}{2}(1-d^2)} m((-1)^{\frac{M}{2} + 1} q^{\frac{M^2}{4} - \frac{aM}{2} + \frac{M}{2}(1-d^2)}, q^{\frac{M^2}{2}}, z') \\
& -\, 2q^{-d^2} \Psi_{\frac{a}{2} - 1}^{\frac{M}{2}}(q^{-d^2}, -1, z'; q^2) \\
& +\, \frac{2}{M} \sum_{j=1}^{M} \zeta_{\frac{M}{2}}^{-\frac{aj}{2}} (1 - \zeta_M^{j}) \Lambda(d, \zeta_{M}^j, z_0, -1).
\end{aligned}
\end{equation}

\item[($ii$)] If $a$ is even and $M$ is odd, then 
\begin{equation} \label{e2}
\begin{aligned}
\overline{D}_d(a, M) + \overline{D}_{d}(a-1, M)  \, &=\,  2(-1)^{\frac{a}{2}} q^{-d^2(\frac{2M-a}{2})^2} m(q^{d^2M(a-M)}, q^{2d^2 M^2}, z') \\
&+\,2(-1)^{\frac{M+1-a}{2}} q^{-d^2 (\frac{M+1-a}{2})^2} m(q^{d^2 M(a-1)}, q^{2d^2 M^2}, z'')\\
&-\,2\Psi_{\frac{2M-a}{2}}^{M}(q^{d^2},-1,z';q^{2d^2}) +2\Psi_{\frac{M+1-a}{2}}^{M}(q^{d^2},-1,z'';q^{2d^2}) \\
& +\, \frac{2}{M} \sum_{j=1}^{M-1} \zeta_{M}^{-aj} (1 - \zeta_M^j) \Lambda(d, \zeta_{M}^j, z_0, -1).
\end{aligned}
\end{equation}

\item[($iii$)] If $a$ is odd and $M$ is odd, then
\begin{equation} \label{e3}
\begin{aligned}
\overline{D}_d(a, M) + \overline{D}_{d}(a-1, M)  \, &=\,  \chi(a=M) \\
& -\, 2 (-1)^{\frac{M-a}{2}} q^{-d^2 (\frac{M-a}{2})^2} m(q^{d^2 Ma} , q^{2d^2M^2}, z')\\
&+\, 2 (-1)^{\frac{a+1}{2}} q^{-d^2 (\frac{2M-a+1}{2})^2} m(q^{d^2 M(a-M-1)}, q^{2d^2M^2},z'')\\
&-\,2\Psi_{\frac{M-a}{2}}^{M}(q^{d^2},-1,z';q^{2d^2}) +2\Psi_{\frac{2M+1-a}{2}}^{M}(q^{d^2},-1,z'';q^{2d^2}) \\
& +\, \frac{2}{M}\sum_{j=1}^{M-1}\zeta_{M}^{-aj} (1-\zeta_{M}^{j}) \Lambda(d, \zeta_{M}^{j}, z_0, -1).
\end{aligned}
\end{equation}
\end{itemize}
\end{theorem}

\begin{theorem} \label{main2} Let $d \geq 2$ be an even integer and $1 \leq a \leq M-1$. For generic $z'$, $z'' \in \mathbb{C}^{*}$, we have the generating function
\begin{equation} \label{e4}
\begin{aligned}
\overline{D}_d(a, M) + \overline{D}_{d}(a-1, M) \, &=\,  \chi(a=1) + 2(-1)^{\frac{da}{2}} q^{-\frac{d^2 a^2}{4}} m((-1)^{1+ \frac{dM}{2}} q^{\frac{d^2}{4}(M^2 - 2Ma)}, q^{\frac{d^2 M^2}{2}}, z') \\
& +\, 2(-1)^{\frac{d}{2}(a-1) + 1} q^{-\frac{d^2}{4}(a^2 - 2a + 1)} m((-1)^{1+\frac{dM}{2}} q^{\frac{d^2}{4}(M^2 - 2M(a-1))}, q^{\frac{d^2 M^2}{2}}, z'') \\
& +\, 2 \Psi_{a}^{M}((-1)^{\frac{d}{2} + 1} q^{\frac{d^2}{4}}, -1, z'; q^{\frac{d^2}{2}})  - 2\Psi_{a-1}^{M}((-1)^{\frac{d}{2} + 1} q^{\frac{d^2}{4}}, -1, z''; q^{\frac{d^2}{2}}) \\
& +\, \frac{2}{M} (-1)^{\frac{d}{2}} q^{-\frac{d^2}{4}} \sum_{j=1}^{M-1} \zeta_M^{j-aj} (1 - \zeta_M^{j}) \Psi_{0}^{\frac{d}{2}}(\zeta_M^{\frac{2j}{d}} q^{1-d}, q, -1; q^2).
\end{aligned}
\end{equation}
\end{theorem}

\begin{remark}
From \cite[page 1154]{jss} and (\ref{Nd}), we have
$$
\overline{N}_{d}(a,M,n) = \overline{N}_{d}(M-a,M,n)
$$
and so by (\ref{dev})
\begin{equation} \label{Dgensymmetry}
\overline{D}_{d}(a, M)  =  \overline{D}_{d}(M-a, M).
\end{equation}
If $d$ and $a$ are odd and $M$ is even, then $\overline{D}_d(a, M) + \overline{D}_d(a-1,M)$ is computed from Theorem~\ref{main1}~(i) using the fact that
\begin{equation*}
\overline{D}_d(a, M) + \overline{D}_d(a-1,M) =  \overline{D}_d(M-a+1, M) + \overline{D}_d(M-a, M)
\end{equation*}
which follows from \eqref{Dgensymmetry}.
\end{remark}

\begin{remark} Let $n \in \mathbb{Z}$. If we take $d=1$ and $z_0=-1$ in Theorem \ref{main1} and use 
\begin{equation} \label{jvan}
j(q^n ; q) = 0,
\end{equation}
then one recovers \cite[Theorem 1.1]{loover}. If we let $d=2$  in Theorem \ref{main2} and employ \cite[Eqs. (2.2a) and (2.2b)]{hm1}
\begin{equation} \label{j1}
j(q^n x; q) = (-1)^n q^{-\binom{n}{2}} x^{-n} j(x;q)
\end{equation}
and 
\begin{equation} \label{j2}
j(x;q) = j(q/x; q) = -x j(x^{-1}; q),
\end{equation}
then we obtain \cite[Theorem 1.2]{loover}.
\end{remark}

Although it is not our primary focus, Theorems \ref{main1} and \ref{main2} in combination with the techniques in \cite[Section 7]{hm2} can be used to deduce the modularity of $\overline{D}_{d}(a,M)$ in arithmetic progressions. This is a ``higher-rank" overpartition analogue of the celebrated results of Bringmann and Ono \cite[Section 7]{zagb}. One can also obtain the mock modularity of $\mathcal{O}_{d}(z;q)$ where $z$ is a root of unity. The idea is to express $\mathcal{O}_{d}(z;q)$ as a sum of pairs of rank deviations (e.g., see (\ref{rewrite})). One then applies Theorem \ref{main1} or \ref{main2} to express $\mathcal{O}_{d}(z;q)$ in terms of Appell--Lerch series and sums of theta quotients (e.g., see (\ref{Main1})). As it is well-known that specializations of Appell--Lerch series are mock theta functions \cite{zwegers}, the mock modularity of $\mathcal{O}_{d}(z;q)$ follows. Thus, our results imply a more explicit version of \cite[Theorems 1.1 and 1.2]{jss}. Finally, Theorems \ref{main1} and \ref{main2} provide a general framework in which one can recover any dissection of $\mathcal{O}_{d}(z;q)$ where $z$ is a root of unity and any generalized rank difference identity. See Section 5 for an example. 

The paper is organized as follows. In Section~2, we recall the required background on Appell--Lerch series and prove two crucial formulas (see Proposition \ref{key}) which express a normalized version of $\mathcal{O}_{d}(z;q)$ in terms of Appell--Lerch series and (\ref{genlam}) if $d$ is odd and in terms of Appell--Lerch series and (\ref{Psikndef}) if $d$ is even. We also give some other key properties of (\ref{j}) and establish $3$-dissections of certain eta quotients, both of which will be beneficial in Section~5. In Section~3, we prove Theorems \ref{main1} and \ref{main2}. In Section~4, we explain why there is no loss in generality in considering pairs of generalized rank deviations in Theorems \ref{main1} and \ref{main2}. In Section~5, we provide an application of Theorem \ref{main1} by computing the $3$-dissection of $\mathcal{O}_{3}(\zeta_3;q)$ (cf. \cite[Theorem 1.3]{jss}).

Finally, we comment that the combinatorial interpretation of $\overline{N}_{d}(m,n)$ involves a certain weighted count of buffered Frobenius representations \cite[Theorem 1.3]{morrill}. It is still an open problem to find a proper combinatorial definition of the ``$M_{d}$-rank" for overparitions \cite[Section~6]{morrill}.

\section{Preliminaries}

\subsection{Appell--Lerch series and $\overline{S}_{d}(z;q)$}

We first state two key results concerning Appell--Lerch series. The first relates two such series with different generic parameters $z_1$ and $z_0$ \cite[Theorem~3.3]{hm1} while the second is an orthogonality result \cite[Theorem~3.9]{hm1}. Recall (\ref{al}), (\ref{delta}) and (\ref{Psikndef}).

\begin{lemma} \label{switch} For generic $x$, $z_0$ and $z_1 \in \mathbb{C}^{*}$,
\begin{equation}
m(x,q,z_1) - m(x,q,z_0) = \Delta(x,z_1, z_0;q).
\end{equation}
\end{lemma}

\begin{lemma} \label{orthog} Let $n$ and $k$ be integers with $0 \leq k < n$. Then
\begin{equation} 
\sum_{t=0}^{n-1} \zeta_n^{-kt} m(\zeta_n^t x, q, z) = n q^{-\binom{k+1}{2}} (-x)^k m(-q^{\binom{n}{2} - nk} (-x)^n, q^{n^2}, z') + n \Psi_{k}^{n}(x,z,z';q).
\end{equation}
\end{lemma}

We also require \cite[Eqs. (3.2b), (3.2e) and (3.3)]{hm1} 
\begin{equation} \label{flip1}
m(x,q,z) = x^{-1} m(x^{-1}, q, z^{-1}),
\end{equation}
\begin{equation} \label{flip2}
m(x,q,z) =  x^{-1} - x^{-1} m(qx,q,z),
\end{equation}
\begin{equation} \label{eval}
m(q,q^2,-1) = \frac{1}{2}
\end{equation}
and \cite[Eq. (4.7)]{hm1}
\begin{equation} \label{htom}
\frac{1}{j(q;q^2)} \sum_{n \in \mathbb{Z}} \frac{(-1)^nq^{n^2+n}}{1-xq^n} = -x^{-1}m(x^{-2}q,q^2,x). 
\end{equation}
In addition, we often use the following well-known fact. Let $s \in \mathbb{Z}$. Then
\begin{equation} \label{sim}
\sum_{j=0}^{n-1} \zeta_n^{sj}= \begin{cases}
 n & \text{if $s \equiv 0 \pmod{n}$}, \\
 0 & \text{otherwise.}
\end{cases}
\end{equation}
To explicitly compute the generalized rank deviations (\ref{dev}), we use the following formula. As the proof is similar to that of \cite[Proposition 2.3]{loover}, we omit it. Let
\begin{equation} \label{sddef}
\overline{S}_{d}(z;q) := (1+z) \mathcal{O}_{d}(z;q).
\end{equation}

\begin{proposition} We have
\begin{equation} \label{overkey}
\frac{1}{M} \sum_{j=1}^{M-1} \zeta_{M}^{-aj} \overline{S}_{d}(\zeta_M^{j}; q) = \overline{D}_{d}(a,M) + \overline{D}_{d}(a-1,M).
\end{equation}
\end{proposition}

We now need to express $\overline{S}_{d}(z;q)$ in terms of Appell--Lerch series and sums of theta quotients. Recall~(\ref{genlam}).

\begin{proposition} \label{key} Let $z$, $z_0$ and $z' \in \mathbb{C}^{*}$ be generic. For $d$ odd, we have
\begin{equation} \label{sdodd}
\overline{S}_{d}(z;q) = (1-z) \left( 1 - 2m(z^{-2} q^{d^2}, q^{2d^2}, z') + 2 \Lambda(d, z, z_0, z') \right).
\end{equation}
For $d$ even, we have
\begin{equation} \label{sdeven}
\overline{S}_{d}(z;q) = (1-z) \left( -1 + 2m((-1)^{\frac{d}{2} + 1} z q^{\frac{d^2}{4}}, q^{\frac{d^2}{2}}, z') + 2(-1)^{\frac{d}{2}} z q^{-\frac{d^2}{4}} \Psi_{0}^{\frac{d}{2}}(z^{\frac{2}{d}} q^{1-d}, q, z'; q^2) \right).
\end{equation}
\end{proposition}
\begin{proof} 
By (\ref{gen}), the fact that
$$
\frac{1}{(1-zq^{dn})(1-z^{-1} q^{dn})} = \frac{-z^2}{(1-z^2)(1-zq^{dn})} + \frac{1}{(1-z^2)(1-z^{-1} q^{dn})}
$$
and some simplifications, we have
\begin{equation} \label{gen1}
\mathcal{O}_{d}(z;q) = \frac{1-z}{1+z} \left( 1 + \frac{2z}{j(q;q^2)} \sum_{n \in \mathbb{Z}} \frac{(-1)^n q^{n^2 + dn}}{1-zq^{dn}} \right).
\end{equation}

For $d$ odd, we let $n \to n - \frac{d-1}{2}$ in (\ref{gen1}) to obtain
\begin{equation} \label{step1}
\mathcal{O}_{d}(z;q) = \frac{1-z}{1+z} \left( 1 + \frac{2z (-1)^{\frac{d-1}{2}} q^{-\frac{(d-1)^2}{4} - \frac{d-1}{2}}}{j(q;q^2)} \sum_{n \in \mathbb{Z}} \frac{(-1)^n q^{n^2 + n}}{1-zq^{dn - \frac{d(d-1)}{2}}} \right).
\end{equation}
An application of
\begin{equation} \label{expand}
\frac{1}{1-x^d} = \frac{1}{d} \sum_{t=0}^{d-1} \frac{1}{1-\zeta_d^{t} x}
\end{equation}
to (\ref{step1}) yields
\begin{equation} \label{step2}
\mathcal{O}_{d}(z;q) = \frac{1-z}{1+z} \left( 1 + \frac{2z (-1)^{\frac{d-1}{2}} q^{-\frac{(d-1)^2}{4} - \frac{d-1}{2}}}{d j(q;q^2)} \sum_{t=0}^{d-1} \sum_{n \in \mathbb{Z}} \frac{(-1)^n q^{n^2 + n}}{1 - \zeta_d^t z^{\frac{1}{d}} q^{-\frac{d-1}{2} + n}} \right).
\end{equation}
Taking $x=\zeta_d^{t} z^{\frac{1}{d}} q^{-\frac{d-1}{2}}$ in (\ref{htom}) turns (\ref{step2}) into 
\begin{equation} \label{step2.5}
\mathcal{O}_{d}(z;q) = \frac{1-z}{1+z} \left(1 - \frac{2z (-1)^{\frac{d-1}{2}} q^{-\frac{(d-1)^2}{4}} z^{-\frac{1}{d}}}{d} \sum_{t=0}^{d-1} \zeta_d^{-t} m(\zeta_{d}^{-2t} z^{-\frac{2}{d}} q^d, q^2, \zeta_d^t z^{\frac{1}{d}} q^{-\frac{d-1}{2}}) \right).
\end{equation}
We now let $x=\zeta_d^{-2t} z^{-\frac{2}{d}} q^d$, $q \to q^2$ and $z_1 = \zeta_d^t z^{\frac{1}{d}} q^{-\frac{d-1}{2}}$ in (\ref{switch}) to express (\ref{step2.5}) as
\begin{equation} \label{step3}
\begin{aligned}
\mathcal{O}_{d}(z;q) = \frac{1-z}{1+z} \Biggl(1 - \frac{2z (-1)^{\frac{d-1}{2}} q^{-\frac{(d-1)^2}{4}} z^{-\frac{1}{d}}}{d} & \Biggl ( \sum_{t=0}^{d-1} \zeta_d^{-t} m(\zeta_d^{-2t} z^{-\frac{2}{d}} q^d, q^2, z_0) \\
& + \sum_{t=0}^{d-1} \zeta_{d}^{-t} \Delta(\zeta_d^{-2t} z^{-\frac{2}{d}} q^d, \zeta_d^{t} z^{\frac{1}{d}} q^{-\frac{d-1}{2}}, z_0; q^2) \Biggr ) \Biggr).
\end{aligned}
\end{equation}
Finally, we apply (\ref{orthog}) with $n=d$, $k=\frac{d-1}{2}$, $q \to q^2$ and $x=z^{-\frac{2}{d}} q^d$ to (\ref{step3}), simplify and appeal to (\ref{genlam}). By (\ref{sddef}), this yields (\ref{sdodd}). 

For $d$ even, we let $n \to n - \frac{d}{2}$ in (\ref{gen1}) to obtain
\begin{equation} \label{oldstep4}
\mathcal{O}_{d}(z;q) = \frac{1-z}{1+z} \left( 1 + \frac{2z (-1)^{\frac{d}{2}} q^{-\frac{d^2}{4}}}{j(q;q^2)} \sum_{n \in \mathbb{Z}} \frac{(-1)^n q^{n^2}}{1-zq^{dn - \frac{d^2}{2}}} \right).
\end{equation}
Using (\ref{expand}), (\ref{oldstep4}) becomes
\begin{equation} \label{step5}
\mathcal{O}_{d}(z;q) = \frac{1-z}{1+z} \left( 1 + \frac{4z (-1)^{\frac{d}{2}} q^{\frac{-d^2}{4}}}{d} \sum_{t=0}^{\frac{d}{2}-1} m(\zeta_{\frac{d}{2}}^t z^{\frac{2}{d}} q^{1-d}, q^2, q) \right).
\end{equation}
Applying (\ref{orthog}) with $n=\frac{d}{2}$, $k=0$, $q \to q^2$ and $x=z^{\frac{2}{d}} q^{1-d}$ to (\ref{step5}) yields
\begin{equation*} \label{step6}
\begin{aligned}
\mathcal{O}_{d}(z;q)  = \frac{1-z}{1+z} & \Biggl( 1 + 2z (-1)^{\frac{d}{2}} q^{-\frac{d^2}{4}} m((-1)^{\frac{d}{2} + 1} z q^{-\frac{d^2}{4}}, q^{\frac{d^2}{2}}, z') \\
& + 2(-1)^{\frac{d}{2}} z q^{-\frac{d^2}{4}} \Psi_{0}^{\frac{d}{2}}(z^{\frac{2}{d}} q^{1-d}, q, z'; q^2) \Biggr). 
\end{aligned}
\end{equation*}
By (\ref{flip2}) and (\ref{sddef}), (\ref{sdeven}) follows.
\end{proof}

\begin{remark}
If we take $d=1$ and $z_0=z'=-1$ in (\ref{sdodd}), then use (\ref{genlam}) combined with (\ref{j2}) and the $z_1=-1$, $z_0=z$ case of (\ref{switch}), we recover \cite[Eq. (3.1)]{loover}. If we take $d=2$ and $z'=q$ in (\ref{sdeven}) and then use (\ref{Psikndef}) and (\ref{jvan}), we obtain \cite[Eq. (3.3)]{loover}.
\end{remark}

\subsection{Identities with $j(z;q)$ and $3$-dissections}

In order to facilitate an application of Theorem \ref{main1} in Section 5, we begin with some identities, the first two of which in the next result follow from (\ref{j}) and (\ref{j2}) while the rest appear in \cite[Eqs. (2.2f), (2.4a), (2.4b), (2.4c) and (2.4f)]{hm1}. Moreover, we frequently use without mention
\begin{gather*}
j(q;q^2) = \frac{J_1^2}{J_2}, \quad j(q;q^3) = J_1, \quad
j(q;q^6) = \frac{J_1 J_6^2}{J_2 J_3},\\
j(-1;q) = 2\frac{J_2^2}{J_1},\quad
j(-q;q^3) = \frac{J_2J_3^2}{J_1J_6}\quad
\text{and} \quad
j(-q;q^6) = \frac{J_2^2J_3J_{12}}{J_1J_4J_6}.
\end{gather*}

\begin{proposition}For generic $x,y,z\in\mathbb C^{*}$, we have
\begin{equation} \label{jnew1}
j(xq^{-1}; q^2) j(x^{-1} q^2; q^2)
=
x^{2} q^{-1} j(x^{-1}; q) \frac{J_2^2}{J_1},
\end{equation}
\begin{equation} \label{jnew2}
\frac{j(-x;q)}{j(-x^2q ; q^2) j(x;q)} = -\frac{j(-x^{-1}; q)}{j(-x^{-2} q; q^2) j(x^{-1}; q)},
\end{equation}
\begin{equation} \label{MH12f}
j(z;q) = \sum_{k=0}^{n-1} (-1)^kq^{k(k-1)/2}z^k j((-1)^{n+1}q^{n(n-1)/2+nk}z^n;q^{n^2}),
\end{equation}
\begin{equation} \label{MH14a}
j(qx^3;q^3) + xj(q^2x^3;q^3)
= J_1\frac{j(x^2;q)}{j(x;q)},
\end{equation}
\begin{equation} \label{MH14b}
j(x;q)j(y,q) = j(-xy;q^2)j(-qx^{-1}y;q^2) -xj(-xyq;q^2)j(-x^{-1}y;q^2),
\end{equation}
\begin{equation} \label{MH14c}
\frac{j(y;q)}{j(-y;q)} -\frac{j(x;q)}{j(-x;q)} = 2x \frac{j(y/x;q^2)j(qxy;q^2)}{j(-x;q)j(-y;q)}
\end{equation}
and
\begin{equation} \label{MH14e}
\frac{j(zx;q)}{j(x;q)}=\frac{J_n^3j(z;q)}{J_1^3j(x^n;q^n)} \sum_{k=0}^{n-1} x^k \frac{j(zx^nq^k;q^n)} {j(zq^k;q^n)}.
\end{equation}
\end{proposition}

Next, we state results which can be found in \cite[Eqs. (1.6) and (1.7)]{AHadvances} or deduced from~\eqref{j}, \eqref{MH12f} and~\eqref{MH14a}.

\begin{proposition} If $w$ is a primitive third root of unity and $x \in \mathbb{C}^{*}$ is generic, then
\begin{equation} \label{AHw1}
j(w; q)
= (1-w)J_3,
\end{equation}
\begin{equation} \label{AHw2}
j(-w; q)
=(1+w)\frac{J_1^2 J_6}{J_2 J_3},
\end{equation}
\begin{equation} \label{AHwnew} 
j(-wq;q^2) 
=\frac{J_1J_4J_6^2}{J_2J_3J_{12}},
\end{equation}
\begin{equation} \label{AHwnew2a}
j(-wq;q^3)
= J_9 \biggl(\frac{j(q^2;q^9)}{j(-q;q^9)} -w^2 q \frac{j(q^8;q^9)}{j(-q^4;q^9)} \biggr),
\end{equation}
\begin{equation}\label{AHwnew3a}
j(-wq;q^6)
= J_{18}
\biggl( \frac{j(q^{10};q^{18})}{j(-q^5;q^{18})} +wq\frac{j(q^{14};q^{18})}{j(-q^7;q^{18})} \biggr)
\end{equation}
and
\begin{equation} \label{prodw}
j(x;q) j(x w; q) j(x w^2 ; q)
=\frac{J_1^3}{J_3} j(x^3;q^3).
\end{equation}
\end{proposition}

Before proceeding, recall that given a $q$-series $F(q)$ with integral powers, its $N$-dissection is a representation of the form
\begin{equation*}
F(q) = \sum_{k=0}^{N-1} q^k F_{k}(q^k)
\end{equation*}
where $F_k(q)$ is a series in $q$ with integral powers. The study of $N$-dissections ostensibly began in Ramanujan's lost notebook \cite{lost} and has now become its own extensive literature. As it is not possible to give a comprehensive overview of this area here, we only mention the classical result of Atkin and Swinnerton-Dyer on $5$ and $7$ dissections of the rank of partitions \cite{asd} and the recent impressive work \cite{boro}
wherein the $11$-dissection of the deviations of the rank and crank modulo $11$ are obtained. Subsequent applications of \cite{boro} include new proofs of crank equalities \cite{garvan}, crank--crank inequalities \cite{bg, ekin}, congruences for $p(n)$ modulo $11$ \cite{asd} and linear rank congruences \cite{ah} as well as new results on rank and rank--crank inequalities, new congruences for rank and crank moments and for Andrews' smallest parts function. The final collection of results concerns $3$-dissections of certain eta quotients. The following two identities are \cite[Eq. (3.9)]{bo} and the $a=-q$, $b=-q^5$, $n=3$ case of \cite[Entry 31, page 48]{berndt}, respectively.

\begin{proposition} We have
\begin{equation}  \label{3dis1}
\frac1{J_1^3} = W_0(q^3) + qW_1(q^3) + q^2W_2(q^3)
\end{equation}
where
\begin{equation*}
W_0(q) := 
\frac{J_3^9}{J_1^{12}} \biggl(\frac{1}{w^2(q)} +8 q w(q) +16 q^2 w^4(q) \biggr),
\end{equation*}
\begin{equation*}
W_1(q) :=
\frac{J_3^9}{J_1^{12}}\biggl( \frac{3}{w(q)} +12 q w^2(q) \biggr),
\qquad W_2(q) :=
9\frac{J_3^9}{J_1^{12}}
\end{equation*}
and
\begin{equation*}
w(q) :=
\frac{J_1J_6^3}{J_2J_3^3}.
\end{equation*}
\end{proposition}

\begin{proposition} We have
\begin{equation}  \label{3dis2}
\frac{J_1J_6}{J_2J_3^2}
= f_0(q^3) + qf_1(q^3) + q^2f_2(q^3)
\end{equation}
where
\begin{equation*}
f_0(q) :=
\frac{j(q^7;q^{18})}{J_1J_2},
\qquad f_1(q) :=
-\frac{j(q^5;q^{18})}{J_1J_2} \qquad \text{and} \qquad f_2(q) := 
-q\frac{j(q;q^{18})}{J_1J_2}.
\end{equation*}
\end{proposition}

The following 3-dissections appear to be new. As the proofs are comparable, we only give details once. 

\begin{proposition} We have
\begin{equation} \label{3dis3}
\frac{J_2^4J_8}{J_1J_4^3}
= g_0(q^3)+qg_1(q^3)+q^2g_2(q^3)
\end{equation}
where
\begin{equation*}
g_0(q) := {\frac{J_{1}\*J_{2}^{2}\*\*J_{8}^{2}\*J_{12}^{2}}{{J_{4}^{5}J_{24}}}},
\qquad g_1(q) := {\frac{J_{2}^{7}\*J_{3}\*J_{8}^{2}\*J_{12}^{4}}{J_{1}^{2}\*J_{4}^{7}\*J_{6}^{3}\*J_{24}}} \qquad \text{and} 
\qquad g_2(q) = {-2
\frac{J_2^2\*J_6^2\*J_8^3}
{J_3\*J_4^{5}}}.
\end{equation*}
\end{proposition}
\begin{proof}
If we multiply both sides of (\ref{3dis3}) by
$\displaystyle {q^{12}\frac{J_{1}J_{4}^{3}J_{12}^{6}J_{18}^{8}J_{36}^{2}}
{J_{2}^{4}J_{8}}}$, then the claim is equivalent to
\bea\label{eq:etadis}
\eta^6(q^{12})\eta^8(q^{18})\eta^2(q^{36})
&=
\frac{
\eta(q)\eta(q^{3})\eta^3(q^{4})\eta^2(q^{6})\eta(q^{12})\eta^8(q^{18})\eta^2(q^{24})\eta^4(q^{36})}
{\eta^4(q^{2})\eta(q^{8})\eta(q^{72})}\\[5pt]
&+
\frac{
\eta(q)\eta^3(q^{4})\eta^7(q^{6})\eta(q^{9})\eta^5(q^{18})\eta^2(q^{24})\eta^6(q^{36})}
{\eta^4(q^{2})\eta^2(q^{3})\eta(q^{8})\eta(q^{12})\eta(q^{72})}
\\[5pt]
&-2
\frac{
\eta(q)\eta^3(q^{4})\eta^2(q^{6})\eta(q^{12})\eta^{10}(q^{18})\eta^3(q^{24})\eta^2(q^{36})}
{\eta^4(q^{2})\eta(q^{8})\eta(q^{9})}
\eea
where $\eta(q) := q^{1/24} J_1$.
Using \cite[Propositions~5.9.2 and 5.9.3]{cs}, we see that both sides of~\eqref{eq:etadis} are holomorphic modular forms of level $72$ and weight $8$. Since the corresponding space of modular forms has Sturm bound $97$ \cite[Corollary 5.6.14]{cs}, it suffices to check that the first $97$ coefficients in the $q$-series expansion of \eqref{eq:etadis} agree. This completes the proof of~\eqref{3dis3}.
\end{proof}

\begin{proposition} We have
\begin{equation}  \label{3dis4}
\frac{J_2^3}{J_1 J_8}=
h_0(q^3)+qh_1(q^3)+q^2h_2(q^3)
\end{equation}
where
\begin{equation*}
h_0(q) := 
\frac{J_{4}^{4}\*J_{6}^{2}}
{J_{2}\*J_{3}\*J_{8}^{3}}, \qquad
h_1(q) := \frac{J_{1}\*J_{4}\*J_{6}\*J_{24}}{J_{8}^{2}\*J_{12}} \qquad \text{and} \qquad
h_2(q) := -\frac{J_{2}^{5}\*J_{3}\*J_{12}\*J_{24}}{J_{1}^{2}\*J_{4}\*J_{6}^{2}\*J_{8}^{2}}.
\end{equation*}
\end{proposition}

\begin{proposition} We have
\begin{equation}  \label{3dis5}
\frac{J_2}{J_4^2}=
I_0(q^3) + qI_1(q^3)+q^2I_2(q^3)
\end{equation}
where
\begin{equation*}
I_0(q) :=
\frac{J_2^2J_6^3}{J_4^6}, \qquad
I_1(q) :=
q\frac{J_2^4J_{12}^6}{J_4^8J_6^3}
\qquad\text{and}\qquad
I_2(q) :=-\frac{J_2^3 J_{12}^3}{J_4^7}.
\end{equation*}
\end{proposition}

\section{Proofs of Theorems \ref{main1} and \ref{main2}}

\begin{proof}[Proof of Theorem \ref{main1}] We follow the strategy in \cite{loover}. Consider the left-hand side of (\ref{overkey}) and take $z'=-1$ in (\ref{sdodd}). This yields

\begin{align} \label{genform}
\displaystyle \frac{1}{M} \sum_{j=1}^{M-1} \zeta_{M}^{-aj} \overline{S}_{d}(\zeta_M^{j}; q) & = \displaystyle \frac{1}{M} \sum_{j=1}^{M-1} \zeta_{M}^{-aj} (1 - \zeta_M^j)(1 - 2 m(\zeta_{M}^{-2j} q^{d^2}, q^{2d^2}, -1) + 2 \Lambda(d, \zeta_{M}^j, z_0, -1)) \nonumber \\
& = \frac{1}{M} \Biggl( \sum_{j=0}^{M-1} \zeta_{M}^{-aj} - \sum_{j=0}^{M-1} \zeta_{M}^{-(a-1)j} \Biggr) - \frac{2}{M} \Biggl ( \sum_{j=0}^{M-1} \zeta_{M}^{-aj} m(\zeta_{M}^{-2j} q^{d^2}, q^{2d^2}, -1) \nonumber \\
& - \sum_{j=0}^{M-1} \zeta_{M}^{-(a-1)j} m(\zeta_{M}^{-2j} q^{d^2}, q^{2d^2}, -1) \Biggr) \nonumber \\
& + \frac{2}{M} \sum_{j=1}^{M} \zeta_{\frac{M}{2}}^{-\frac{aj}{2}} (1 - \zeta_M^{j}) \Lambda(d, \zeta_{M}^j, z_0, -1).
\end{align}

To prove (\ref{e1}), we first use (\ref{sim}) to observe that the first two sums in the second line of (\ref{genform}) equal $0$ if $a<M$ and $1$ if $a=M$. For the third sum, we split it into two further sums. We then reindex the resulting second sum by $j \to \frac{M}{2} + j$, use that $\zeta_{M}^{-2} = \zeta_{\frac{M}{2}}^{-1}$, write $a=2t$ where $t>0$, apply (\ref{flip1}) and take $k=t-1$, $n=\frac{M}{2}$, $z=-1$, $x=q^{-d^2}$ and $q \to q^{2d^2}$ in (\ref{orthog}) to obtain

\begin{align*}
\sum_{j=0}^{M-1} \zeta_{M}^{-aj} m(\zeta_{M}^{-2j} q^{d^2}, q^{2d^2}, -1) & = 2 \sum_{j=0}^{\frac{M}{2} - 1} \zeta_{M}^{-aj} m(\zeta_{\frac{M}{2}}^{-j} q^{d^2}, q^{2d^2}, -1) \\
& = M (-1)^{\frac{a}{2} - 1} q^{-\frac{a^2}{4} + \frac{a}{2}(1-d^2)} m((-1)^{\frac{M}{2} + 1} q^{\frac{M^2}{4} - \frac{aM}{2} + \frac{M}{2}(1-d^2)}, q^{\frac{M^2}{2}}, z') \\
& \qquad \qquad \qquad \qquad \qquad \qquad \qquad \quad + q^{-d^2} M \Psi_{\frac{a}{2} - 1}^{\frac{M}{2}}(q^{-d^2}, -1, z'; q^2).
\end{align*}
For the fourth sum, we similarly split it into two further sums, then reindex the resulting second sum by $j \to \frac{M}{2} + j$. We then use that $\zeta_{M}^{\frac{M}{2}} = -1$ to obtain $0$. In total, this yields (\ref{e1}).

To prove (\ref{e2}), we first use (\ref{sim}) to note that the first two sums in the second line of (\ref{genform}) equal $0$. For the third sum, we first use that if $\zeta_M$ is a primitive $M$-th root of unity, then so is $\zeta_M^2$ as $M$ is odd and then take $k=\frac{2M-a}{2}$, $n=M$, $z=-1$, $x=q^{d^2}$ and $q \to q^{2d^2}$ in (\ref{orthog}) to obtain
\begin{align*}
\sum_{j=0}^{M-1} \zeta_{M}^{-aj} m(\zeta_{M}^{-2j} q^{d^2}, q^{2d^2}, -1) & = M q^{-d^2(\frac{2M-a}{2})^2} (-1)^{\frac{2M-a}{2}} m(q^{d^2M(a-M)}, q^{2d^2 M^2}, z') \\
& \qquad \qquad \qquad \qquad \qquad \qquad \qquad \quad + M \Psi_{\frac{2M-a}{2}}^{M}(q^{d^2}, -1, z' ;q^{2d^2}).
\end{align*}
For the fourth sum, we take $k=\frac{M+1-a}{2}$, $n=M$, $z=-1$, $x=q^{d^2}$ and $q \to q^{2d^2}$ in (\ref{orthog}) to obtain

\begin{align*}
\sum_{j=0}^{M-1} \zeta_{M}^{-(a-1)j} m(\zeta_{M}^{-2j} q^{d^2}, q^{2d^2}, -1) & = M q^{-d^2(\frac{a+M-1}{2})^2} (-1)^{\frac{a+M-1}{2}} m(q^{d^2(M-aM)}, q^{2 d^2 M^2}, z'') \\
& \qquad \qquad \qquad \qquad \qquad \qquad \quad + M \Psi_{\frac{a+M-1}{2}}^{M}(q^{d^2},-1,z''; q^{2d^2}).
\end{align*}

\noindent In total, this yields (\ref{e2}).

Finally, to prove (\ref{e3}), we first use (\ref{sim}) to see that the first two sums in the second line of (\ref{genform}) equal $0$ unless $a=M$, in which case the sum is $M$.  For the third sum, we take $k=\frac{M-a}{2}$, $n=M$, $z=-1$, $x=q^{d^2}$ and $q \to q^{2d^2}$ in (\ref{orthog}) to obtain
\begin{align*}
\sum_{j=0}^{M-1} \zeta_{M}^{-aj} m(\zeta_{M}^{-2j} q^{d^2}, q^{2d^2}, -1) & = M q^{-d^2 (\frac{M-a}{2})^2} (-1)^{\frac{M-a}{2}} m(q^{d^2 aM}, q^{2d^2 M^2}, z') \\
& \qquad \qquad \qquad \qquad \qquad \qquad \qquad \quad + M \Psi_{\frac{M-a}{2}}^{M}(q^{d^2}, -1, z'; q^{2d^2}). 
\end{align*}
For the fourth sum, we take $k=\frac{2M-a+1}{2}$, $n=M$, $z=-1$, $x=q^{d^2}$ and $q \to q^{2d^2}$ in (\ref{orthog}) to obtain
\begin{align*}
\sum_{j=0}^{M-1} \zeta_{M}^{-(a-1)j} m(\zeta_{M}^{-2j} q^{d^2}, q^{2d^2}, -1) & = M q^{-d^2(\frac{2M-a+1}{2})^2} (-1)^{\frac{2M-a+1}{2}} m(q^{d^2M(a-M-1)}, q^{2 d^2 M^2}, z'') \\
& \qquad \qquad \qquad \qquad \qquad \qquad \qquad + M \Psi_{\frac{2M-a+1}{2}}^{M}(q^2,-1,z''; q^{2d^2}).
\end{align*}
In total, this yields (\ref{e3}).
\end{proof}

\begin{proof}[Proof of Theorem \ref{main2}]
Consider the left-hand side of (\ref{overkey}) and take $z'=-1$ in (\ref{sdeven}). This yields
\begin{align} \label{genform2}
\displaystyle \frac{1}{M} \sum_{j=1}^{M-1} \zeta_{M}^{-aj} \overline{S}_{d}(\zeta_M^{j}; q) & = \displaystyle \frac{1}{M} \sum_{j=1}^{M-1} \zeta_{M}^{-aj} (1 - \zeta_M^j) \Biggl(-1 + 2 m((-1)^{\frac{d}{2} + 1} \zeta_{M}^{j} q^{\frac{d^2}{4}}, q^{\frac{d^2}{2}}, -1) \nonumber \\
& \qquad \qquad \qquad \qquad \qquad \qquad \qquad + 2 (-1)^{\frac{d}{2}} \zeta_{M}^{j} q^{-\frac{d^2}{4}} \Psi_{0}^{\frac{d}{2}}(\zeta_M^{\frac{2j}{d}} q^{1-d}, q, -1; q^2) \Biggr) \nonumber \\
& = -\frac{1}{M} \Biggl( \sum_{j=0}^{M-1} \zeta_{M}^{-aj} - \sum_{j=0}^{M-1} \zeta_{M}^{-(a-1)j} \Biggr) \nonumber \\
& + \frac{2}{M} \Biggl ( \sum_{j=0}^{M-1} \zeta_{M}^{-aj} m((-1)^{\frac{d}{2} + 1} \zeta_{M}^{j} q^{\frac{d^2}{4}}, q^{\frac{d^2}{2}}, -1) \nonumber \\
& - \sum_{j=0}^{M-1} \zeta_{M}^{-(a-1)j} m((-1)^{\frac{d}{2} + 1} \zeta_{M}^{j} q^{\frac{d^2}{4}}, q^{\frac{d^2}{2}}, -1) \Biggr) \nonumber \\
& + \frac{2}{M} (-1)^{\frac{d}{2}} q^{-\frac{d^2}{4}} \sum_{j=1}^{M-1} \zeta_{M}^{j-aj} (1 - \zeta_M^{j}) \Psi_{0}^{\frac{d}{2}}(\zeta_M^{\frac{2j}{d}} q^{1-d}, q, -1; q^2).
\end{align}

To prove (\ref{e4}), we first use (\ref{sim}) to observe that the first two sums in the second line of (\ref{genform2}) equal $0$ unless $a=1$, in which case we obtain $1$. For the third sum, we take $k=a$, $n=M$, $z=-1$, $x=(-1)^{\frac{d}{2} + 1} q^{\frac{d^2}{4}}$ and $q \to q^{\frac{d^2}{2}}$ in (\ref{orthog}) to obtain
\begin{align*}
\sum_{j=0}^{M-1} \zeta_{M}^{-aj} m(\zeta_{M}^{j} (-1)^{\frac{d}{2} + 1} q^{\frac{d^2}{4}} , q^{\frac{d^2}{2}}, -1) & = M (-1)^{\frac{da}{2}} q^{-\frac{d^2 a^2}{4}} m((-1)^{1+ \frac{dM}{2}} q^{\frac{d^2}{4}(M^2 - 2Ma)}, q^{\frac{d^2 M^2}{2}}, z') \nonumber \\
& \qquad \qquad \qquad \qquad \qquad + M \Psi_{a}^{M}((-1)^{\frac{d}{2} + 1} q^{\frac{d^2}{4}}, -1, z'; q^{\frac{d^2}{2}}).
\end{align*}
For the fourth sum, we take $k=a-1$, $n=M$, $z=-1$, $x=(-1)^{\frac{d}{2} + 1} q^{\frac{d^2}{4}} $ and $q \to q^{\frac{d^2}{2}}$ in (\ref{orthog}) to obtain
\begin{align*}
& \sum_{j=0}^{M-1} \zeta_{M}^{-(a-1)j} m((-1)^{\frac{d}{2} + 1} \zeta_M^{j} q^{\frac{d^2}{4}}, q^{\frac{d^2}{2}}, -1) \nonumber \\
& \qquad \qquad \qquad \qquad = M (-1)^{\frac{d}{2}(a-1) + 1} q^{-\frac{d^2}{4}(a^2 - 2a + 1)} m((-1)^{1+\frac{dM}{2}} q^{\frac{d^2}{4}(M^2 - 2M(a-1))}, q^{\frac{d^2 M^2}{2}}, z'') \nonumber \\
& \qquad \qquad \qquad \qquad \qquad \qquad \qquad \qquad \qquad \qquad \qquad \qquad + M \Psi_{a-1}^{M}((-1)^{\frac{d}{2} + 1} q^{\frac{d^2}{4}} , -1, z''; q^{\frac{d^2}{2}}).
\end{align*}
In total, this yields (\ref{e4}).
\end{proof}

\section{Single generalized rank deviations}
In this section, we briefly discuss why there is no loss in generality in considering pairs of generalized rank deviations in Theorems \ref{main1} and \ref{main2}. For $M$ odd, we have
\begin{equation*}
\overline{D}_d \Bigl(\frac{M+1}{2}, M \Bigr) + \overline{D}_d \Bigl(\frac{M-1}{2}, M \Bigr) =  2\overline{D}_{d} \Bigl(\frac{M-1}{2}, M \Bigr)
\end{equation*}
and so Theorem \ref{main1} can be used to find a formula for any single $\overline{D}_{d}(a, M)$. Precisely, we prove the following result.

\begin{proposition} Let $M$ be odd. For $0 \leq n \leq \frac{M+1}{2}$, we have
\begin{equation*}
\begin{aligned}
\overline{D}_d \Bigl(\frac{M+1}{2} + n, M \Bigr) & = \frac{1}{2} \Biggl( \sum_{i=0}^{n} \left[ \overline{D}_{d}\Bigl(\frac{M+1}{2} - n + 2i, M\Bigr) + \overline{D}_d \Bigl(\frac{M+1}{2} - n + 2i-1, M \Bigr) \right] \\
& - \sum_{i=0}^{n-1} \left[ \overline{D}_d\Bigl( \frac{M+1}{2} - n + 2i+1, M \Bigr) + \overline{D}_d\Bigl( \frac{M+1}{2} - n + 2i, M \Bigr) \right]\Biggr).
\end{aligned}
\end{equation*}
\end{proposition}

\begin{proof}
We begin by letting
\begin{equation*}
\overline{D}_n := \sum_{i=0}^{n} \left[ \overline{D}_d\Bigl( \frac{M+1}{2} + i, M \Bigr) + \overline{D}_{d}\Bigl( \frac{M-1}{2} - i, M \Bigr) \right].
\end{equation*}
By (\ref{Dgensymmetry}), we have
\begin{equation*} \label{sta}
\overline{D}_n = 2 \sum_{i=0}^{n} \overline{D}_d \Bigl( \frac{M+1}{2} + i, M \Bigr) 
\end{equation*}
and so 
\begin{equation}  \label{dnsym}
\overline{D}_d\Bigl(\frac{M+1}{2} + n, M \Bigr) = \frac{\overline{D}_n - \overline{D}_{n-1}}{2}.
\end{equation}
After reindexing, one obtains 
\begin{equation} \label{dnsym2}
\overline{D}_n = \sum_{i=0}^{n} \left[ \overline{D}_{d}\Bigl(\frac{M+1}{2} - n + 2i, M\Bigr) + \overline{D}_d \Bigl(\frac{M+1}{2} - n + 2i-1, M \Bigr) \right] .
\end{equation}
Combining (\ref{dnsym}) and (\ref{dnsym2}) yields the result.
\end{proof}

For $M$ even, one can use the following result. 

\begin{proposition} Let $M$ be even, $0 \leq n \leq M-1$ and $z_0$, $z' \in \mathbb{C}^{*}$ be generic. For $d$ odd, we have
\begin{equation} \label{Meven1}
\begin{aligned}
\overline{D}_d(n,M) & = \frac{(-1)^n}{M} \mathcal{O}_{d}(-1;q) + \frac{1}{M} \sum_{k=1}^{\frac{M}{2} - 1} \left ( \frac{1-\zeta_M^{k}}{1+\zeta_{M}^k} \right ) \left(\zeta_M^{-kn} + \zeta_M^{kn} \right)  \\
& \qquad \qquad \qquad \qquad \qquad \qquad \qquad \quad \times \left (1 - 2m(\zeta_M^{-2k} q^{d^2}, q^{2d^2}, z') + 2 \Lambda(d, \zeta_M^k, z_0, z') \right).
\end{aligned}
\end{equation}
For $d$ even, we have
\begin{equation} \label{Meven2}
\begin{aligned}
\overline{D}_d(n,M) & = \frac{(-1)^n}{M} \mathcal{O}_{d}(-1;q) \\
& + \frac{1}{M} \sum_{k=1}^{\frac{M}{2} - 1}  \left ( \frac{1-\zeta_M^{k}}{1+\zeta_{M}^k} \right ) \left(\zeta_M^{-kn} + \zeta_M^{kn} \right) \Biggl(-1 + 2m((-1)^{\frac{d}{2} + 1} \zeta_M^k q^{\frac{d^2}{4}}, q^{\frac{d^2}{2}}, z') \\
& \qquad \qquad \qquad \qquad \qquad \qquad \qquad \qquad  + 2(-1)^{\frac{d}{2}} \zeta_M^k q^{-\frac{d^2}{4}} \Psi_{0}^{\frac{d}{2}}(\zeta_M^{\frac{2k}{d}} q^{1-d}, q, z'; q^2) \Biggr).
\end{aligned}
\end{equation}
\end{proposition}

\begin{proof} 
From (\ref{gen}) and (\ref{sim}), one deduces 
\begin{equation} \label{forallM}
\overline{D}_{d}(n,M) = \frac{1}{M} \sum_{k=1}^{M-1} \mathcal{O}_d(\zeta_M^{k}; q) \zeta_M^{-kn}
\end{equation}
for all $M$. For $M$ even, we remove the $k=\frac{M}{2}$ term from (\ref{forallM}), split it into two further sums, then reindex the resulting second sum by $k \to M - k$. This yields
\begin{equation} \label{left}
\overline{D}_{d}(n, M) = \frac{(-1)^n}{M} \mathcal{O}_{d}(-1;q) + \frac{1}{M} \sum_{k=1}^{\frac{M}{2} -1} \left(\zeta_M^{-kn} + \zeta_M^{kn} \right) \mathcal{O}_d(\zeta_M^k; q).
\end{equation}
Finally, (\ref{Meven1}) and (\ref{Meven2}) follow from applying (\ref{sddef}), (\ref{sdodd}) and (\ref{sdeven}) to (\ref{left}).
\end{proof}

\section{An application of Theorem \ref{main1}}
As an application of Theorem \ref{main1}, we compute the $3$-dissection of $\mathcal{O}_{3}(\zeta_3; q)$ (cf. \cite[Theorem 1.3]{jss}). In principle, one can reduce the number of eta quotients appearing in our result. We do not pursue this further. Let 
\begin{equation} \label{ABCDEFG}
\begin{aligned}
A &:= j(-q^{12}; q^{27}), \quad B := q j(-q^{21}; q^{27}), \quad C := q^2 j(-q^3; q^{27}), \quad D := \frac{j(q^{60}; q^{108})}{j(-q^{30}; q^{108})}, \\
& \qquad \quad E := q^6 \frac{j(q^{84}; q^{108})}{j(-q^{42}; q^{108})}, \quad F := \frac{j(q^{24}; q^{108})}{j(-q^{12}; q^{108})}, \quad G := q^{12} \frac{j(q^{96}; q^{108})}{j(-q^{48}; q^{108})}, 
\end{aligned}
\end{equation}
\begin{equation} \label{GN}
\mathcal G_N := \sum_{\substack{k,\, l,\, m \, \in \, \{0,1,2\} \\ k+l+m \, \equiv \, N \;(\text{mod}\; 3)}} q^{k+l+m}g_k(q^3)W_l(q^3)f_m(q^3)
\end{equation}  
and
\begin{equation} \label{HN}
\mathcal H_N := \sum_{\substack{k,\, l,\, m \, \in \, \{0,1,2\} \\ k+l+m \, \equiv \, N \;(\text{mod}\; 3)}} q^{k+l+m}h_k(q^3)W_l(q^3)f_m(q^3)
\end{equation}    
where $N=0$, $1$ or $2$ and the functions $W_i(q)$, $f_i(q)$, $g_i(q)$ and $h_i(q)$ are given in (\ref{3dis1})--(\ref{3dis4}). Also, recall the functions $I_i(q)$ in (\ref{3dis5}).

\begin{theorem} \label{main3} We have
\begin{equation} \label{3dis}
\mathcal{O}_{3}(\zeta_3; q) = \overline{\mathcal{B}}_0(q^3) + q \mathcal{B}_1(q^3) + q^2 \mathcal{B}_2(q^3)
\end{equation}
where 
\begin{equation*} \label{constant}
\begin{aligned}
\overline{\mathcal{B}}_0(q^3) & := 6q^{-36} m(q^{-27}, q^{162}, -1) \\
& \qquad -\frac{3}{2} q^{-9}
\frac
  {J_{18}J_{27}J_{108}J_{162}^5}
  {J_{36}^2J_{54}J_{81}J_{324}^3}
\biggl(
\frac{j(q^{27}; q^{162})}{j(-q^{27};q^{162})}
+
\frac{j(q^{81}; q^{162})}{j(-q^{81};q^{162})}
\biggr)
 + \mathcal{B}_0(q^3)
\end{aligned}
\end{equation*} 
and 
\begin{equation*} \label{rest}
\begin{aligned}
\mathcal B_N(q^3) &
:= 3q^3 \frac{J_6^3J_{9}J_{108}}
{J_3J_{18}{J_{36}}} q^NI_N(q^3) \\
&\qquad + \frac{J_3^2J_6^2J_{36}}{J_{12}J_{18}^2} 
\Biggl(
\frac{J_3^3J_{12}^2J_{18}^2J_{72}J_{108}^2}
{J_6^4J_9J_{24}J_{36}J_{54} J_{216}}
\sum_{\substack{k, \, l \, \in \, \{0,1,2\} \\ k+l \, \equiv \, N \;(\mathrm{mod}\; 3)}}
  q^{k+l}g_k(q^3)W_l(q^3)
\\
&\qquad\qquad  -2q^2
\frac{J_{12}^2J_{108}}{J_6J_{24}}
\Bigl(
\bigl( 2AD - AE \bigr)
\mathcal G_{N+1}
-\bigl(BD + BE \bigr)
\mathcal G_{N}
+\bigl(2CE-CD\bigr)
\mathcal G_{2+N}
\Bigr)
\\
&\qquad\qquad  +  
 q^5 \frac
{J_3^3J_{18}J_{24}J_{36}^2J_{216}}
{J_6^3J_{9}J_{12}J_{72}J_{108}}
\sum_{\substack{k,\, l \, \in \, \{0,1,2\} \\ k+l \, \equiv \, N+1 \;(\mathrm{mod}\; 3)}}
  q^{k+l}h_k(q^3)W_l(q^3)
\\
&\qquad\qquad -2q
\frac{J_{24}J_{108}}{J_{12}}
\Bigl(
\bigl(2AG+ AF\bigr)
\mathcal H_{N+2}
-
\bigl(2BF + BG\bigr)
\mathcal H_{N+1}
-
\bigl(CG-CF\bigr)
\mathcal H_{N}
\Bigr)
\Biggr)
\end{aligned}
\end{equation*}
for $N=0$, $1$ and $2$.
\end{theorem}

\begin{proof}[Proof of Theorem \ref{main3}] 
We begin by decomposing $\mathcal{O}_{3}(\zeta_3; q)$ as follows. By (\ref{gen})--(\ref{Nd}) and (\ref{Dgensymmetry}), we have
\begin{equation} \label{rewrite}
\begin{aligned}
\mathcal{O}_{3}(\zeta_3; q) & = \sum_{n \geq 0} \sum_{s=0}^{2} \overline{N}_3(s, 3, n) \zeta_3^s q^n \\
& = \sum_{n \geq 0} \left( \overline{N}_3(0, 3, n) + \zeta_3 \overline{N}_3(1,3,n) + \zeta_3^2 \overline{N}_3(2,3,n) \right) q^n \\
& = \sum_{n \geq 0} \left(  \overline{N}_3(0, 3, n) -  \overline{N}_3(2, 3, n) \right) q^n \\
& =  \overline{D}_3(3,3) + \overline{D}_3(2,3)  - \left( \overline{D}_3(2,3) + \overline{D}_3(1,3) \right)
\end{aligned}
\end{equation} 
where we have used that $\overline{N}_3(1, 3, n) =  \overline{N}_3(2, 3, n)$ and $1 + \zeta_3 + \zeta_3^2 = 0$. We now take $d=3$, $a=M=3$, $z' = z'' = z_0 = -1$ in (iii) and $d=3$, $a=2$, $M=3$, $z' = z'' = z_0 = -1$ in (ii) of Theorem~\ref{main1}, respectively, and simplify using (\ref{flip1}), (\ref{eval}) and the fact that
\begin{equation*}
\Psi_{0}^{3}(q^9, -1, -1; q^{18}) = 0
\end{equation*}
to obtain from (\ref{rewrite})
\begin{equation} \label{Main1}
\begin{aligned}
\mathcal{O}_{3}(\zeta_3; q) &= 6q^{-36} m(q^{-27}, q^{162}, -1)  - 2 \zeta_3 \Lambda(3, \zeta_3, -1, -1) - 2\zeta_3^2 \Lambda(3, \zeta_3^2, -1, -1) \\
& +  4 \Psi_{2}^{3}(q^9, -1, -1; q^{18}) - 2 \Psi_{1}^{3}(q^9, -1, -1; q^{18}). \\
\end{aligned}
\end{equation}
By (\ref{genlam}), (\ref{j1}), (\ref{j2}) and (\ref{jnew1}), we have
\begin{equation} \label{step20}
\begin{aligned} 
- 2 \zeta_3 \Lambda(3, \zeta_3, -1, -1) & = -2q^{-1}\zeta_3^{5/3}
\Psi^3_{1}(\zeta_3^{-2/3}q^3,-1,-1; q^2)\\
&\quad + \frac {2\zeta_3^{5/3} J_2^3}{3qj(-1; q^2)}
\sum_{t=0}^2 \zeta_3^{-t}
\frac{
j(-\zeta_3^{t+1/3}/q;q^2)
j(-\zeta_3^{-t-1/3}q^2;q^2)}
{j(\zeta_3^{t+1/3}/q;q^2)
j(\zeta_3^{-t-1/3}q^2;q^2)
j(-\zeta_3^{-2t-2/3}q^3;q^2)
}
\\ 
& = -\frac{\zeta_3 J_2J_6J_{18}^4}
{2J_{4}^2 J_{36}^{2} j(-\zeta_3q^9;q^{18})}
\biggl(
\frac{j(\zeta_3q^{15};q^{18})}{j(-\zeta_3q^{15};q^{18})}
+\frac{j(\zeta_3q^{21};q^{18})}{j(-\zeta_3q^{21};q^{18})}
\biggr) \\ 
&\qquad \qquad \qquad +
\frac{\zeta_3J_2^4}{3J_4^2}
\sum_{t=0}^2 \frac{j(-\zeta_3^{t+1/3}q;q)}{j(\zeta_3^{t+1/3}q;q) j(-\zeta_3^{2t+2/3}q;q^2)}.
\end{aligned}
\end{equation}
Similarly,
\begin{equation} \label{step30}
\begin{aligned} 
- 2\zeta_3^2 \Lambda(3, \zeta_3^2, -1, -1) & =
\frac{\zeta_3^2 J_2J_6J_{18}^4}
{2J_{4}^2 J_{36}^{2} j(-\zeta_3^2q^9;q^{18})}
\biggl(\frac{j(\zeta_3q^{15};q^{18})}{j(-\zeta_3q^{15};q^{18})}
+\frac{j(\zeta_3q^{21};q^{18})}{j(-\zeta_3q^{21};q^{18})}
\biggr)
\\
& \qquad \qquad -\frac{\zeta_3^2J_2^4}{3J_4^2}
\sum_{t=0}^2
\frac{j(-\zeta_3^{t+2/3};q)}
{j(\zeta_3^{t+2/3};q) j(-\zeta^{2t+4/3}q;q^2)}
\end{aligned}
\end{equation}
after applying (\ref{j1}), (\ref{j2}) and (\ref{jnew2}). Using (\ref{j1}), (\ref{j2}), the $x=\zeta_3^2 q^{15}$ and $y=\zeta_3 q^{15}$ case of (\ref{MH14c}), (\ref{AHw2}) and (\ref{prodw}), one confirms 
\begin{equation} \label{combine}
\begin{aligned}
\frac{j(\zeta_3q^{15};q^{18})}{j(-\zeta_3q^{15};q^{18})} + \frac{j(\zeta_3q^{21};q^{18})} {j(-\zeta_3q^{21};q^{18})}
&= -2q^3 \zeta_3^2(1-\zeta_3^2) 
\frac
{J_{6}^2J_{9}^2J_{36}^2J_{54}^2}
{J_{3}J_{18}^6J_{27}}.
\end{aligned}
\end{equation}
Hence, taking the sum of the first terms on the right-hand sides of (\ref{step20}) and (\ref{step30}), respectively, combined with (\ref{AHwnew}) and (\ref{combine}) yields after simplification
\begin{equation} \label{inter}
\begin{aligned}
-(\zeta_3-\zeta_3^2)\frac{ J_2J_6J_{18}^4}
{2J_{4}^2 J_{36}^{2}j(-\zeta_3q^9;q^{18})}
\biggl(
\frac{j(\zeta_3q^{15};q^{18})}{j(-\zeta_3q^{15};q^{18})}
+\frac{j(\zeta_3q^{21};q^{18})}{j(-\zeta_3q^{21};q^{18})}
\biggr) 
= 3q^3
\frac{J_2J_6^3J_{9}J_{108}}
{J_3J_{4}^2J_{18}J_{36}}.
\end{aligned}
\end{equation}
It now remains to consider the sum of the second terms on the right-hand sides of (\ref{step20}) and (\ref{step30}), respectively. This is
\begin{equation} \label{eq:sumla}
(\zeta_3-\zeta_3^2)
\frac{J_2^4}{3J_4^2}
\sum_{t=0}^2 \frac{j(-\zeta_3^{t+1/3}q;q)}{j(\zeta_3^{t+1/3}q;q) j(-\zeta_3^{2t+2/3}q;q^2)}.
\end{equation}
We now apply (\ref{j2}) to the summand of (\ref{eq:sumla}), then take $x=\zeta_3^{t+1/3}$, $z=-1$ and $n=3$ in (\ref{MH14e}) to obtain
\begin{equation} \label{sumla2}
-(\zeta_3-\zeta_3^2)
\frac
  {2J_2^6J_3^3}
  {3J_1^4J_4^2j(\zeta_3;q^3)}
\sum_{k=0}^2
\frac{j(-\zeta_3q^k;q^3)}{j(-q^k;q^3)}
\sum_{t=0}^2
\frac{\zeta_3^{(t+1/3)k}}{j(-\zeta_3^{2(t+1/3)}q;q^2)}.
\end{equation}
By (\ref{prodw}),
\begin{equation} \label{middle}
j(-\zeta_{3}^{2/3} q; q^2) j(-\zeta_3^{8/3} q; q^2) j(-\zeta_3^{14/3} q; q^2)
= \frac{J_2^3J_3J_{12}J_{18}^2}
  {J_6^2J_9J_{36}}
\end{equation}
and so using (\ref{AHw1}) and (\ref{middle}) turns (\ref{sumla2}) into
\begin{equation} \label{sumla3}
\begin{aligned}
&-\zeta_3 \frac
{2J_2^3J_3^2J_6^2J_{36}}
{3J_1^4J_4^2J_{12}J_{18}^2}
\sum_{k=0}^2
\frac{j(-\zeta_3q^k;q^3)}{j(-q^k;q^3)}
\sum_{t=0}^2
{\zeta_3^{(t+1/3)k}}
{j(-\zeta_3^{2(t+4/3)}q;q^2)}
{j(-\zeta_3^{2(t+7/3)}q;q^2)}.
\end{aligned}
\end{equation}
Note that
\begin{equation*}
\begin{aligned}
& j(-\zeta_3^{2t+8/3} q;q^2) j(-\zeta_3^{2t+5/3}q;q^2) \\
& \qquad \qquad \qquad \qquad = \frac{J_2J_8J_{12}^2}{J_4J_6J_{24}}
j(-\zeta_3^{4(t+1/3)}q^2;q^4)
-\zeta_3^{2t + 2/3} q
\frac{J_4^2J_{24}}{J_8J_{12}}
j(-\zeta_3^{4(t+1/3)}q^4;q^4)
\end{aligned}
\end{equation*}
by \eqref{MH14b}, \eqref{AHw1} and \eqref{AHwnew} and so (\ref{sumla3}) equals
\begin{equation} \label{newstep2}
\begin{aligned}
-\frac
{2J_2^3J_3^2J_6^2J_{36}}
{3J_1^4J_4^2J_{12}J_{18}^2}
&\Biggl(\zeta_3
\frac{J_2J_8J_{12}^2}{J_4J_6J_{24}}
\sum_{k=0}^2
\frac{j(-\zeta_3q^k;q^3)}{j(-q^k;q^3)}
\sum_{t=0}^2
{\zeta_3^{(t+1/3)k}}
j(-\zeta_3^{4(t+1/3)}q^2;q^4)\,
\\
&-q \frac{J_4^2J_{24}}{J_8J_{12}}
\sum_{k=0}^2
{\zeta_3^{(k+5)/3}}
\frac{j(-\zeta_3q^k;q^3)}{j(-q^k;q^3)}
\sum_{t=0}^2
{\zeta_3^{(k+2)t}}
j(-\zeta_3^{(t+4/3)}q^4;q^4)\Biggr).
\end{aligned}
\end{equation}
We now simplify the first line in (\ref{newstep2}). By (\ref{j}) and (\ref{sim}),
\begin{equation} \label{newstep3}
\begin{aligned}
\sum_{t=0}^2 \zeta_3^{(t+1/3)k} j(-\zeta_3^{4(t+1/3)}q^2;q^4) & = \sum_{n\in\mathbb Z} q^{2n^2} \sum_{t=0}^2 \zeta_3^{(t+1/3)(4n+k)} \\
& = 3 \sum_{\substack{n \,\equiv \, 2k \;(\text{mod}\; 3)}} q^{2n^2} \zeta_3^{(4n+k)/3} \\
& = 3 \sum_{s\in\mathbb Z} q^{36s(s-1)/2+(24k + 18)s+8k^2} \zeta_3^{s} \\
&= 3 q^{8k^2}j(-\zeta_3q^{24k + 18};q^{36}).
\end{aligned}
\end{equation}
First, applying (\ref{AHwnew}) to (\ref{newstep3}), the $k=0$ term in the first line of (\ref{newstep2}) equals
\begin{equation} \label{simpk0}
-\frac
  {3J_3^3J_{18}^2J_{72}J_{108}^2}
  {2J_6^3J_9J_{36}J_{54}J_{216}}.
\end{equation}
Next, applying (\ref{AHwnew3a}) to (\ref{newstep3}) and the $z=-\zeta_3 q$, $q \to q^3$ and $n=3$ case of (\ref{MH12f}) and simplifying, the $k=1$ term in the first line of (\ref{newstep2}) equals
\begin{equation}  \label{simpk1}
3q^2 \frac{J_1J_6J_{108}}{J_2J_3^2}
\left( AD-BE-CD+CE +\zeta_3(AE+BD-BE-CD) \right)
\end{equation}
where $A$, $B$, $C$, $D$ and $E$ are given in (\ref{ABCDEFG}). Finally, applying (\ref{j1}), (\ref{j2}) and (\ref{AHwnew3a}) to (\ref{newstep3}) and simplifying, the $k=2$ term in the first line of (\ref{newstep2}) equals
\begin{equation} \label{simpk2}
3q^2 \frac{J_1J_6J_{108}}{J_2J_3^2}
(AD-AE-BD+CE +\zeta_3(-AE-BD+BE-CD)).
\end{equation}
Combining (\ref{simpk0})--(\ref{simpk2}) and simplifying, the first line of (\ref{newstep2}) equals
\begin{equation} \label{newterm}
-\frac{J_2^4J_3^2J_6J_8J_{12}J_{36}}{J_1^4J_4^3J_{18}^2J_{24}}
\biggl(-
\frac
  {J_3^3J_{18}^2J_{72}J_{108}^2}
  {J_6^3J_9J_{36}J_{54}J_{216}}
+2q^2 \frac{J_1J_6J_{108}}{J_2J_3^2}
(2AD-AE-BD-BE-CD+2CE)
\biggr).
\end{equation}
We now simplify the second line in (\ref{newstep2}). Similar to (\ref{newstep3}), we find
\begin{equation} \label{newstep3b}
\sum_{t=0}^2
\zeta_3^{(k+2)t}
j(-\zeta_3^{t+4/3};q^4)
= 3\zeta_3^{(5k-5)/3} 
q^{2k^2-6k+4}
j(-\zeta_3q^{36-12k};q^{36}).
\end{equation}
Applying (\ref{AHw2}) to (\ref{newstep3b}), the $k=0$ term in the second line of (\ref{newstep2}) equals
\begin{equation} \label{ssimpk0}
3q^{4}
\frac{J_3^3J_{18}J_{36}^2J_{216}}
{2J_6^3J_9J_{72}J_{108}}.
\end{equation}
Next, applying (\ref{j2}) and (\ref{AHwnew2a}) to (\ref{newstep3b}) and (again) the $z=-\zeta_3 q$, $q \to q^3$ and $n=3$ case of (\ref{MH12f}) and simplifying, the $k=1$ term in the second line of (\ref{newstep2}) becomes
\begin{equation} \label{ssimpk1}
3 \frac{J_1J_6J_{108}}{J_2J_3^2}
(-AF-AG+BF+CG
+\zeta_3(-AF-BG+CF+CG))
\end{equation}
where $F$ and $G$ are given in (\ref{ABCDEFG}). Finally, one can similarly show that the $k=2$ term in the second line of (\ref{newstep2}) is
\begin{equation} \label{ssimpk2}
3 \frac{J_1J_6J_{108}}{J_2J_3^2}
(-AG+BF+BG-CF + \zeta_3 (AF+BG-CF-CG)).
\end{equation}
Combining (\ref{ssimpk0})--(\ref{ssimpk2}) and simplifying, the second line of (\ref{newstep2}) equals
\begin{equation} \label{newterm2}
-q\frac{J_2^3J_3^2J_6^2J_{24}J_{36}}{J_1^2J_4^2J_8J_{12}^2J_{18}^2}
\biggl(q^{4}
\frac{J_3^3J_{18}J_{36}^2J_{216}}
{J_6^3J_9J_{72}J_{108}}
+ 2\frac{J_1J_6J_{108}}{J_2J_3^2} \bigl(-2AG + 2BF + BG - AF - CF + CG \bigr) \biggr).
\end{equation}
As \eqref{newstep2} is the sum of (\ref{newterm}) and (\ref{newterm2}), we obtain
\begin{equation} \label{newstep4}
\begin{aligned}
&
\frac
{J_2^3J_3^2J_6^2J_{36}}
{J_1^4J_4^2J_{12}J_{18}^2}
\Biggl(
\frac{J_2J_8J_{12}^2}{J_4J_6J_{24}}
\biggl(
\frac
  {J_3^3J_{18}^2J_{72}J_{108}^2}
  {J_6^3J_9J_{36}J_{54}J_{216}}
-2q^2 \frac{J_1J_6J_{108}}{J_2J_3^2}
(2AD-AE-BD-BE-CD+2CE)
\biggr)
\\
&\qquad\qquad +
q\frac{J_4^2J_{24}}{J_8J_{12}}
\biggl(q^{4}
\frac{J_3^3J_{18}J_{36}^2J_{216}}
{J_6^3J_9J_{72}J_{108}}
+ 2\frac{J_1J_6J_{108}}{J_2J_3^2} \bigl(-2AG + 2BF + BG - AF - CF + CG \bigr) \biggr) \Biggr).
\end{aligned}
\end{equation}
Note that (\ref{newstep4}) can be rearranged as
\begin{equation} \label{laststep}
\begin{aligned}
\frac{J_3^2J_6^2J_{36}}{J_{12}J_{18}^2} 
\Biggl(
&
\frac{J_3^3J_{12}^2J_{18}^2J_{72}J_{108}^2}
{J_6^4J_9J_{24}J_{36}J_{54} J_{216}}
\frac{J_2^4J_8}{J_1J_4^3}
\frac 1 {J_1^3}\\
& -2q^2 \frac{J_{12}^2J_{108}}{J_6J_{24}}
\bigl( 2AD - AE - BD - BE - CD + 2CE \bigr)
\frac{J_2^4J_8}{J_1J_4^3}
\frac 1 {J_1^3}
\frac{J_1J_6}{J_2J_3^2}\\
& +  
 q^5 
\frac
{J_3^3J_{18}J_{24}J_{36}^2J_{216}}
{J_6^3J_{9}J_{12}J_{72}J_{108}}
\frac{J_2^3}{J_1J_8}
\frac 1 {J_1^3}
\\
& +2q
\frac{J_{24}J_{108}}{J_{12}} 
\bigl(-2AG + 2BF + BG - AF - CF + CG \bigr)
\frac{J_2^3}{J_1J_8}
\frac 1 {J_1^3}
\frac{J_1J_6}{J_2J_3^2}
\Biggr).
\end{aligned}
\end{equation}
Moreover, using (\ref{Psikndef}), (\ref{j1}), (\ref{j2}) and simplifying, one can check 
\begin{equation} \label{step4}
\begin{aligned}
&4 \Psi_{2}^{3}(q^9, -1, -1; q^{18}) - 2 \Psi_{1}^{3}(q^9, -1, -1; q^{18})  \\
&\qquad \qquad \qquad \qquad \qquad \qquad \qquad =
-\frac{3}{2} q^{-9}
\frac
  {J_{18}J_{27}J_{108}J_{162}^5}
  {J_{36}^2J_{54}J_{81}J_{324}^3}
\biggl(
\frac{j(q^{27}; q^{162})}{j(-q^{27};q^{162})}
+
\frac{j(q^{81}; q^{162})}{j(-q^{81};q^{162})}
\biggr).
\end{aligned}
\end{equation}
We now insert (\ref{3dis1})--(\ref{3dis4}) into (\ref{laststep}) and (\ref{3dis5}) into (\ref{inter}) and combine with (\ref{step4}). After substituting the resulting expressions into (\ref{Main1}) and recalling (\ref{GN}) and (\ref{HN}), we arrive at (\ref{3dis}). This proves the result.
\end{proof}

\section*{Acknowledgements}
The authors were partially funded by the Irish Research Council Advanced Laureate Award IRCLA/2023/1934. The authors would like to thank Amanda Folsom and the referee for their helpful comments and suggestions
which improved the exposition of our paper.

\end{document}